\documentclass[12pt]{article}
\usepackage[english]{babel}
\usepackage[numbers,sort,compress]{natbib}
\usepackage[T1]{fontenc}
\usepackage{amsthm,amsmath,amssymb}
\usepackage[mathscr]{eucal}
\usepackage[utf8]{inputenc}
\usepackage{etoolbox}
\usepackage{latexsym}
\usepackage{graphicx}
\usepackage{geometry}
\usepackage{longtable}
\usepackage{array}     
\usepackage{pgffor}    
\usepackage{caption}
\newtheorem{theorem}{Theorem}
\newtheorem{corollary}{Corollary}
\newtheorem{proposition}{Proposition}
\newtheorem{lemma}{Lemma}
\newtheorem{definition}{Definition}
\newtheorem{conjecture}{Conjecture}

\def \T{\textup{T}}
\def \Res{\textup{Res}}
\def \diag{\textup{diag}}
\begin{document}
	\title{A general formula for walk determinants of rooted products with applications to DGS-graph constructions}
\author{Wei Wang$^a$\thanks{Corresponding author. Email address: wangwei.math@gmail.com} \quad Jie Shen$^a$ \quad Lihuan Mao$^b$ \\[2mm]
{\footnotesize  $^a$School of Mathematics, Physics and Finance, Anhui Polytechnic University, Wuhu 241000, China}\\
{\footnotesize  $^b$School of Mathematics and Data Science, Shaanxi University of Science and Technology, Xi’an	710021, China
}
}
\date{}

\maketitle
\begin{abstract}
	For an $n$-vertex graph $G$, and a rooted graph $H^{(v)}$ with $v$ as the root, the rooted product graph $G\circ H^{(v)}$ is obtained from $G$ and $n$ copies of $H$ by identifying the root of the $i$th copy of $H$  with the $i$th vertex of $G$ for each $i$. As a refinement of the controllability  criterion  of $G\circ H^{(v)}$ obtained recently by Shan and Liu (2025), we obtain an explicit formula for the determinant of the walk matrix of $G\circ H^{(v)}$. Furthermore, for an important family of graphs $\mathcal{F}$ that are determined by their generalized spectrum (DGS), we introduce the concept of $\mathcal{F}$-preservers and provide a sufficient condition for a rooted graph to be an $\mathcal{F}$-preserver. A list of $\mathcal{F}$-preservers of small order is provided, which leads to many new infinite families of DGS-graphs using rooted products.
\end{abstract}
		\noindent\textbf{Keywords:} walk matrix; rooted product graph;  generalized spectrum\\
	
	\noindent\textbf{Mathematics Subject Classification:} 05C50

\section{Introduction}
Let $G$ be a simple graph with $n$ vertices, and let $M(G)$ be a real symmetric matrix associated with $G$. Throughout this paper, $M(G)$ may refer to  adjacency matrix $A(G)$, signless Laplacian matrix $Q(G)$,  and the $A_\alpha$ matrix $A_\alpha(G)$. The characteristic polynomial of $M(G)$, defined as $\det (xI-M(G))$, is called the $M$-characteristic polynomial of $G$ and is denoted by $\phi(M(G);x)$.  The $M$-walk matrix of $G$, denoted by $W_M(G)$, or $W(M(G))$, is defined to be the matrix
\begin{equation*}
[e, M(G)e,\ldots, (M(G))^{n-1}e],
\end{equation*}
where  $e$ is the all-ones vector.  This particular kind of matrix has received increasing attention in spectral graph theory, as it has deep connection with many important properties of a graph. For example,  a theorem of Wang \cite{wang2017JCTB} states that if $2^{-\lfloor n/2\rfloor} \det(W_A(G))$ is odd and square-free, then $G$ is determined by its generalized spectrum (DGS). In particular, all graphs $G$ with $\det W_A(G)=\pm 2^{\lfloor n/2\rfloor}$ are DGS. For more studies on the walk matrices of graphs, we refer to \cite{godsil2012,liu2022,wang2021,qiu2019,rourke2016,wang2023Eujc}.

Let $H^{(v)}$ be a rooted graph with a root vertex $v$. The \emph{rooted product graph} of $G$ and $H^{(v)}$, denoted by $G\circ H^{(v)}$ (or simply $G\circ H$ when no confusion arises) is defined as follows: Take $n$ copies of $H^{(v)}$ and identify the root $v$ of the $i$th copy of $H^{(v)}$ with the $i$th vertex of $G$ for each $i$. See Fig.~\ref{c4} for an illustration. 
\begin{figure}
	\centering
	\includegraphics[height=3.6cm]{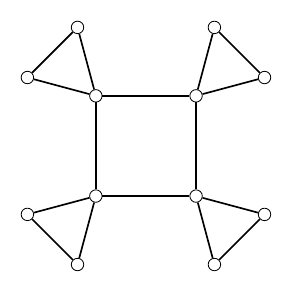}
	\caption{The rooted product graph $C_4\circ C_3$}
	\label{c4}
\end{figure} We are interested in the determinant of the walk matrix of the rooted product graph $G\circ H$. We may assume that $H$ is not the trivial graph $K_1$ as $G\circ K_1=G$. The first nontrivial result on $\det W(G\circ H)$ is obtained by Mao, Liu and Wang \cite{mao2015}, which we  state as the following theorem.

\begin{theorem}[\cite{mao2015}]\label{gp2}
	For any graph $G$,
	$$\det W_A(G\circ P_2)=\pm (\det A(G))(\det W_A(G))^2.$$
\end{theorem}
Similar formulas for $\det W_A(G\circ P_3)$ and $\det W_A(G\circ P_4)$ were obtained later by Mao and Wang \cite{mao2022}, where the end vertex of $P_3$ (or $P_4$) is taken as the root. The general formula for $\det W_A(G\circ P_m)$ was conjectured by Mao and Wang \cite{mao2022} and proved by Wang, Yan and Mao \cite{wym2024}.   Recently,  Yan and Wang \cite{yan2026} further extended these results to any $ P_m^{(\ell)}$, where the vertices of $P_m$ are labeled naturally as $\{1,2,\ldots,m\}$ and  the root $\ell$ can be arbitrary.
\begin{theorem}[\cite{yan2026}]
	For any graph $G$ and integer $m\ge 2$,
	\begin{equation*}
		\det W_A(G\circ P_m^{(\ell)}) =
		\begin{cases}
			\pm (\det A(G))^{\lfloor\frac{m}{2}\rfloor}(\det W_A(G))^m, & \text{if $\gcd(\ell,m+1)=1$,} \\
			0,&\text{otherwise.}
		\end{cases}
	\end{equation*}
\end{theorem} 
In \cite{tian2025},  Tian et al. considered the $Q$-analog of  the problem and proved that 
\begin{equation}\label{q23}
\det W_Q(G\circ P_m)=\pm (\det Q(G))^{m-1} (\det W_Q(G))^2 \text{~for~} m=2,3.
\end{equation}
Yan, Wang and Mao \cite{yan2024} extended Eq.~\eqref{q23} for any $m\ge 2$.

Note that all the current results on $\det W(G\circ H^{(v)})$ are restricted to the case where $H$ is the  path graph. It seems natural to consider the general case, where $H$ is an arbitrary graph.  The main result of this paper is the following theorem, which  establishes a general formula for $\det W_M(G\circ H^{(v)})$.
\begin{theorem}\label{main}
	Let  $G$ be a graph of order $n$ and $H^{(v)}$ be a rooted graph of order $m$. Let $B(\lambda)=M(H)+\lambda D_v$ where $D_v$ is the diagonal matrix $\diag(0,\ldots,0,1,0,\ldots,0)$ with only nonzero entry (which is 1) corresponding to the root $v$. Let $M^{(v)}(H)$ be the matrix obtained from $M(H)$ by deleting the row and the column corresponding to $v$. Then we have
	\begin{equation*}
		\det W_M(G\circ H^{(v)})=\pm \left(\Res(\phi(M(H);x),\phi(M^{(v)}(H);x))\right)^{\frac{n(n-1)}{2}}\cdot\det h(M(G))\cdot\left(\det W_M(G)\right)^m,
	\end{equation*}
	where $h(\lambda)=\det W(B(\lambda))=\det[e, B(\lambda)e,\ldots, (B(\lambda))^{m-1}e]$.
\end{theorem}
A graph $G$ is \emph{$M$-controllable} if $\det W(M(G))\neq 0$. Similarly, an $n\times n$ matrix $B$ is \emph{controllable} if $\det W(B)\neq 0$. Note that for two polynomials $f$ and $g$, the resultant $\Res(f,g)$ is nonzero if and only if $\gcd(f,g)=1$. Also note that for any matrix $M$ and any polynomial $h(x)$, it holds that $\det h(M)=\prod h(\lambda)$, where $\lambda$ runs through all eigenvalues of $M$.  An immediate consequence of Theorem \ref{main} is  the following necessary and sufficient condition for controllability of $G\circ H^{(v)}$, which was obtained by Shan and  Liu \cite{shan}.
\begin{corollary}
	The graph $G\circ H^{(v)}$ is $M$-controllable if and only if the following three conditions hold simultaneously:\\
	\noindent\textup{(i)} $G$ is $M$-controllable;\\
	\noindent\textup{(ii)} $\gcd(\phi(M(H);x),\phi(M^{(v)}(H);x))=1$; and\\
	\noindent\textup{(iii)} for any eigenvalue $\lambda$ of $M(G)$, the matrix $B(\lambda)$ is controllable.
	\end{corollary}
We are most interested in Theorem \ref{main} for the case  where $M(G)$ refers to the adjacency matrix $A(G)$, as in this setting it gives a way to construct larger DGS-graphs from smaller ones, according to the aforementioned theorem of Wang \cite{wang2017JCTB}. It is worthwhile to mention that this construction fails in the setting of generalized $Q$ spectrum; see \cite{tian2025}.

 For convenience, we introduce an important family of DGS-graphs.
\begin{definition}
	For an integer $n$, we define $\mathcal{F}_n$ to be the collection of all $n$-vertex graphs $G$ such that $\det A(G)=\pm 1$ and $\det W_A(G)=\pm 2^{\lfloor n/2\rfloor}$. Define $\mathcal F=\bigcup \mathcal{F}_n$.
	\end{definition}
	We note that $\mathcal{F}_n$ is empty when $n$ is odd, since in this case $\det A(G)\equiv 0\pmod{2}$, contradicting the requirement that  $\det A(G)=\pm 1$. On the other hand,  for all even $n\ge 6$, the families $\mathcal{F}_n$ are nonempty; see  \cite[Fig. 1]{liu2019}.
	
We call a rooted graph $H^{(v)}$ an \emph{$\mathcal{F}$-preserver} if $G\circ H^{(v)}\in \mathcal{F}$ for any $G\in \mathcal{F}$. Since any graph in $\mathcal{F}$ is DGS by the theorem of Wang \cite{wang2017JCTB}, it is straightforward to use $\mathcal{F}$-preservers to construct an infinite family of DGS-graphs. Indeed, let $G$ be any graph in $\mathcal{F}$ and $\{H_i^{(v_i)}\}$ be a sequence of $\mathcal{F}$-preservers, not necessarily distinct. Then all graphs in the family 
$$G, G\circ H_1^{(v_1)}, (G\circ H_1^{(v_1)})\circ H_2^{(v_2)},\ldots$$
belong to the family $\mathcal{F}$ and hence are DGS. The second  main result of this paper is a sufficient condition for a rooted graph to be an $\mathcal{F}$-preserver,  which we state as the following theorem using notation of Theorem \ref{main}.
\begin{theorem}\label{main2}
	A rooted graph $H^{(v)}$ is an $\mathcal{F}$-preserver if the following three conditions hold simultaneously:\\		
		\noindent{\textup{(i)}} $\Res(\phi(A(H);x),\phi(A^{(v)}(H);x))=\pm 1$;\\
		\noindent{\textup{(ii)}}  $\det[e,B(\lambda)e,\ldots,(B(\lambda))^{m-1}e] =\pm \lambda^k$ for some nonnegative integer $k$; and\\
		\noindent{\textup{(iii)}}  either $\det A(H)=\pm 1$ and $\det A^{(v)}(H)=0$, or $\det A(H)=0$ and $\det A^{(v)}(H)=\pm 1$.		
\end{theorem} 
\section{Proofs of Theorems \ref{main} and \ref{main2}}
Let $G$ be an $n$-vertex graph and $H^{(v)}$ be an $m$-vertex rooted graph with root $v$. We define $B(\lambda)=M(H)+\lambda D_v$, where $D_v$ is the diagonal matrix $\diag(0,\ldots,0,1,0,\ldots,0)$ with the number $1$ corresponding to the root $v$. Let $M^{(v)}(H)$ be the matrix obtained from $M(H)$ by deleting the row and the column corresponding to the root vertex $v$. We mention that for adjacency matrix $A(G)$, the matrix $A^{(v)}(H)$ equals $A(H-v)$, where $H-v$ is the vertex-deleted subgraph of $H$. However, the corresponding equality does not hold for $Q(H)$ in general. 

 It was observed that $M(G\circ H^{(v)})=M(H)\otimes I_n+D_v\otimes M(G)$ \cite{shan},  where $\otimes$ denotes the Kronecker product of matrices. The following lemma characterizes the $M$-eigenvalues and $M$-eigenvectors of $G\circ H^{(v)}$, generalizing the well-known result for the  $A$-eigenvalues of $G\circ H^{(v)}$; see \cite{schwenk1974,godsil1978,gutman1980}.

\begin{lemma}[\cite{shan}]\label{evev}
	Let $\lambda_1,\ldots,\lambda_n$ be the eigenvalues of $M(G)$ with corresponding eigenvectors $\xi_1,\ldots,\xi_n$. Let $\mu_i^{(j)}$ $(j\in \{1,\ldots,m\})$ be the eigenvalues of $B(\lambda_i)=M(H)+\lambda_i D_v$ with corresponding eigenvectors $\zeta_{i}^{(j)}$. Denote $\tilde{M}=M(G\circ H^{(v)})$ and $\eta_i^{(j)}=\zeta_i^{(j)}\otimes \xi_i$. Then $\tilde{M} \eta_i^{(j)}=\mu_i^{(j)}\eta_{i}^{(j)}$ for $i\in\{1,\ldots,n\}$ and $j\in\{1,\ldots,m\}$.
\end{lemma}

\begin{proposition}\label{dm}
	$\det M(G\circ H^{(v)})=\det((\det M(H))\cdot I_n+(\det M^{(v)}(H))\cdot M(G)).$
\end{proposition}
\begin{proof}
	Using the notation of Lemma \ref{evev}, we have 
	$$\det B(\lambda_i)=\det (M(H)+\lambda_i D_v)=\det M(H)+\lambda_i\det M^{(v)}(H)$$
	by the multilinearity of the determinant. Since the determinant of a matrix equals the product of all its eigenvalues, we find that
$$\det M(G\circ H^{(v)})=\prod_{i=1}^{n}\prod_{j=1}^m \mu_i^{(j)}=\prod_{i=1}^{n}\det B(\lambda_i)=\prod_{i=1}^{n}(\det M(H)+\lambda_i\det M^{(v)}(H)).$$
Let $M'=(\det M(H))\cdot I_n+(\det M^{(v)}(H))\cdot M(G)$. As $M(G)$ has eigenvalues $\lambda_1,\ldots,\lambda_n$, we see that the eigenvalues of $M'$ are exactly $\det M(H)+\lambda_i\det M^{(v)}(H)$. It follows that 
$$\prod_{i=1}^{n}(\det M(H)+\lambda_i\det M^{(v)}(H))=\det M'.$$ Thus, $\det M(G\circ H^{(v)})=\det M'$, completing the proof.
\end{proof}
The following corollary is immediate from Proposition \ref{dm}.
\begin{corollary}\label{suf} If the rooted graph $H^{(v)}$ satisfies \textup{(i)} $\det A(H)=\pm 1$ and $\det A^{(v)}(H)=0$; or \textup{(ii)} $\det  A(H)=0$ and $\det A^{(v)}(H)=\pm1$,
	then  $\det A(G\circ H^{(v)})=\pm 1$ for any graph $G$ with $\det A(G)=\pm 1$. 
\end{corollary}
\begin{definition}\normalfont{
		Let $f(x)=a_nx^n+a_{n-1}x^{n-1}+\cdots+a_1x+a_0$ and $g(x)=b_mx^m+b_{m-1}x^{m-1}+\cdots+b_1x+b_0$. The resultant of $f(x)$ and $g(x)$, denoted by  $\Res(f(x),g(x))$, is defined to be
		\begin{equation}\label{defres}
		a_n^mb_m^n\prod_{1\le i\le n,1\le j\le m}(\alpha_i-\beta_j),
		\end{equation}
		where $\alpha_i$'s and $\beta_j$'s are the roots (in complex field $\mathbb{C}$) of $f(x)$ and $g(x)$, respectively.
	}
\end{definition}
We note that Eq.~\eqref{defres} can be rewritten as 
\begin{equation}\label{defresequ}
	a_n^m\prod_{1\le i\le n}g(\alpha_i), \text{~or~}  	(-1)^{mn}b_m^n\prod_{1\le j\le m}f(\beta_j).
\end{equation}
We shall use the following basic property of the resultant, which is clear from the second expression in Eq.~\eqref{defresequ}.
\begin{lemma}\label{basicres}
	Let $f(x)=a_nx^n+\cdots+a_0=a_n\prod_{i=1}^n(x-\alpha_i)$ and $g(x)=b_mx^m+\cdots+b_0=b_m\prod_{j=1}^m(x-\beta_j)$. If $m<n$,  then $\Res(f(x)+tg(x),g(x))=\Res(f(x),g(x))$ for any $t\in \mathbb{C}$.
\end{lemma}
Now we turn to the calculation of $\det W(G\circ H^{(v)})$.  The basic tool is the  following formula which gives the determinant of walk matrix $W(M)$ using the  eigenvalues and eigenvectors of $M$. 
\begin{lemma}[\cite{mao2015}]\label{basicW}
	Let $M$ be an $n\times n$ real symmetric matrix. Let $\lambda_i$ be the eigenvalues of $M$ with corresponding eigenvectors $\xi_i$ for $i=1,2,\ldots,n$. Then
	$$\det W(M)= \frac{\prod_{1\le i_1< i_2\le n}(\lambda_{i_2}-\lambda_{i_1})\prod_{1\le i\le n}(e_n^\T \xi_i)}{\det[\xi_1,\xi_2,\ldots,\xi_n]}.$$
\end{lemma}
Let $\Omega=\{(i,j)\colon\,1\le i\le n \text{~and~} 1\le j\le m\}$ with the colexicographical order: $(i_1,j_1)<(i_2,j_2) $ if either $j_1<j_2$, or $j_1=j_2$ and $i_1<i_2$.  The following formula for $\det W(G\circ H^{(v)})$ is an immediate consequence of Lemmas \ref{evev}  and \ref{basicW}.
\begin{corollary}\label{dwt3}
	\begin{equation*}
		\det W(G\circ H^{(v)})= \frac{\prod_{(i_1,j_1)<(i_2,j_2)}(\mu_{i_2}^{(j_2)}-\mu_{i_1}^{(j_1)})\prod_{(i,j)\in \Omega}(e_{mn}^\T \eta_i^{(j)})}{\det[\eta_1^{(1)},\ldots,\eta_n^{(1)};\ldots;\eta_1^{(m)},\ldots,\eta_n^{(m)}]}.
	\end{equation*}
\end{corollary}

\begin{lemma}[\cite{yan2026}]\label{Vanmu}
	\begin{equation*}
		\prod_{(i_1,j_1)<(i_2,j_2)}(\mu_{i_2}^{(j_2)}-\mu_{i_1}^{(j_1)})=\pm	\left(\prod_{i=1}^{n}\prod_{1\le j_1<j_2\le m}\left(\mu_i^{(j_2)}-\mu_i^{(j_1)}\right)\right)\left(\prod_{1\le i_1< i_2\le n}	\prod_{j_2=1}^{m}\prod_{j_1=1}^{m}\left(\mu_{i_2}^{(j_2)}-\mu_{i_1}^{(j_1)}\right)\right).
	\end{equation*}
\end{lemma}
\begin{lemma}\label{ppmu}
	\begin{equation*}
		\prod_{j_2=1}^{m}\prod_{j_1=1}^{m}\left(\mu_{i_2}^{(j_2)}-\mu_{i_1}^{(j_1)}\right)=(\lambda_{i_2}-\lambda_{i_1})^m\Res(\phi(M(H);x),\phi(M^{(v)}(H);x)).
	\end{equation*}
\end{lemma}
\begin{proof}
	Note that $\phi(B(\lambda_i);x)=\det(xI-(M(H)+\lambda_iD_v))=\phi(M(H);x)-\lambda_i\phi(M^{(v)}(H);x)$, which is a monic polynomial with roots $\mu_i^{(j)}$, $j=1,\ldots,m$. We find that
	$$\phi(M(H);x)-\lambda_i\phi(M^{(v)}(H);x)=\prod_{j_1=1}^m(x-\mu_i^{(j_1)}),$$
	and hence
		\begin{equation}\label{phim}
		\phi(M(H);\mu_{i_2}^{(j_2)})-\lambda_{i_1}\phi(M^{(v)}(H);\mu_{i_2}^{(j_2)})=\prod_{j_1=1}^m(\mu_{i_2}^{(j_2)}-\mu_{i_1}^{(j_1)}).
	\end{equation}
	As $\phi(M(H);\mu_{i_2}^{(j_2)})-\lambda_{i_2}\phi(M^{(v)}(H);\mu_{i_2}^{(j_2)})=0$, we see that the left-hand side of Eq.~\eqref{phim} equals 
		$$(\lambda_{i_2}-\lambda_{i_1})\phi(M^{(v)}(H);\mu_{i_2}^{(j_2)}),$$
		implying
		\begin{equation}\label{uu}
			\prod_{j_2=1}^{m}\prod_{j_1=1}^{m}\left(\mu_{i_2}^{(j_2)}-\mu_{i_1}^{(j_1)}\right)=\prod_{j_2=1}^{m}(\lambda_{i_2}-\lambda_{i_1})\phi(M^{(v)}(H);\mu_{i_2}^{(j_2)})=(\lambda_{i_2}-\lambda_{i_1})^m\prod_{j_2=1}^m \phi(M^{(v)}(H);\mu_{i_2}^{(j_2)}).
		\end{equation}
		Finally, by the definition of the resultant, we have
			\begin{eqnarray}
		\prod_{j_2=1}^m \phi(M^{(v)}(H);\mu_{i_2}^{(j_2)})&=&\Res(\phi(M(H);x)-\lambda_{i_2}\phi(M^{(v)}(H);x),\phi(M^{(v)}(H);x))\nonumber \\
			&=&\Res(\phi(M(H);x),\phi(M^{(v)}(H);x)),\nonumber 
		\end{eqnarray}
		where the last equality follows from Lemma \ref{basicres} and the  fact $\deg \phi(M^{(v)}(H);x)<\deg \phi(M(H);x)$.
		This, together with Eq.~\eqref{uu}, completes the proof of Lemma \ref{ppmu}.	
\end{proof}
\begin{lemma}\label{deteta1}
	\begin{equation*}
		\det[\eta_1^{(1)},\ldots,\eta_n^{(1)};\ldots;\eta_1^{(m)},\ldots,\eta_n^{(m)}]= \left(\det[\xi_1,\xi_2,\ldots,\xi_n]\right)^m \prod_{i=1}^{n}\det [\zeta_i^{(1)},\zeta_i^{(2)},\ldots,\zeta_i^{(m)}].
	\end{equation*}
\end{lemma}
\begin{proof}
	Let $E^{(j)}=[\eta_1^{(j)},\eta_2^{(j)},\ldots,\eta_n^{(j)}]$ for $j\in \{1,\ldots,m\}$. Write $\zeta_i^{(j)}=(z^{(j)}_{i,1},z^{(j)}_{i,2},\ldots,z^{(j)}_{i,m})^\T$. Then,  we have
	\begin{eqnarray}\label{ej}
		E^{(j)}&=&\left[\zeta_1^{(j)}\otimes \xi_1,\zeta_2^{(j)}\otimes \xi_2,\ldots,
		\zeta_n^{(j)}\otimes \xi_n \right]\nonumber \\
		&=&\left[\begin{bmatrix}
		z_{1,1}^{(j)}\cdot \xi_1\\
			z_{1,2}^{(j)} \cdot\xi_1\\
			\vdots\\
			z_{1,m}^{(j)}\cdot \xi_1
		\end{bmatrix},\begin{bmatrix}
		z_{2,1}^{(j)} \cdot\xi_2\\
			z_{2,2}^{(j)}\cdot\xi_2\\
			\vdots\\
			z_{2,m}^{(j)} \cdot\xi_2
		\end{bmatrix},\cdots,\begin{bmatrix}
			z_{n,1}^{(j)} \cdot\xi_n\\
			z_{n,2}^{(j)}  \cdot\xi_n\\
			\vdots\\
			z_{n,m}^{(j)} \cdot\xi_n
		\end{bmatrix}\right]\nonumber\\
		&=&\begin{bmatrix}
			[\xi_1,\xi_2,\ldots,\xi_n]\cdot\diag[	z_{1,1}^{(j)},			z_{2,1}^{(j)} ,\ldots,			z_{n,1}^{(j)}]\\
			[\xi_1,\xi_2,\ldots,\xi_n]\cdot\diag[	z_{1,2}^{(j)},			z_{2,2}^{(j)},\ldots,			z_{n,2}^{(j)}]\\
			\vdots\\
			[\xi_1,\xi_2,\ldots,\xi_n]\cdot\diag[	z_{1,m}^{(j)},		z_{2,m}^{(j)},\ldots,			z_{n,m}^{(j)}]
		\end{bmatrix}.
	\end{eqnarray}
	
	Let $U=[\xi_1,\xi_2,\ldots,\xi_n]$ and $T_{k,j}=\diag[z_{1,k}^{(j)},z_{2,k}^{(j)},\ldots,z_{n,k}^{(j)}]$ for $k,j\in\{1,2,\ldots,m\}$. Then we can rewrite Eq.~\eqref{ej} as 
	\begin{equation*}
		E^{(j)}=\begin{bmatrix} U\cdot T_{1,j}\\
			U\cdot T_{2,j}\\
			\vdots\\
			U\cdot T_{m,j}
		\end{bmatrix}=\begin{bmatrix}U&&&\\
			&U&&\\
			&&\ddots&\\
			&&&U\end{bmatrix}\begin{bmatrix}  T_{1,j}\\
			T_{2,j}\\
			\vdots\\
			T_{m,j}
		\end{bmatrix}.
	\end{equation*}
	This implies the following identity:
	\begin{equation}\label{e1n}
		[E^{(1)},E^{(2)},\ldots,E^{(m)}]=\begin{bmatrix}
			U&&&\\
			&U&&\\
			&&\ddots&\\
			&&&U
		\end{bmatrix} \begin{bmatrix}
			T_{1,1}&T_{1,2}&\cdots&T_{1,m}\\
			T_{2,1}&T_{2,2}&\cdots&T_{2,m}\\
			\vdots&\vdots&&\vdots\\
			T_{m,1}&T_{m,2}&\cdots&T_{m,m}
		\end{bmatrix}.
	\end{equation}
	Since each $T_{k,j}$ is a diagonal matrix, one easily sees that the block matrix $(T_{k,j})$ is permutationally similar to the following block diagonal matrix

	\begin{equation*}
		\diag[[\zeta_1^{(1)},\zeta_1^{(2)},\ldots,\zeta_1^{(m)}],[\zeta_2^{(1)},\zeta_2^{(2)},\ldots,\zeta_2^{(m)}],\ldots,[\zeta_n^{(1)},\zeta_n^{(2)},\ldots,\zeta_n^{(m)}]].
	\end{equation*}
	Thus, taking determinants on both sides of Eq.~\eqref{e1n} leads to
	\begin{equation*}
		\det [E^{(1)},E^{(2)},\ldots,E^{(m)}]=(\det U)^m\prod_{1\le i\le n}\det[\zeta_i^{(1)},\zeta_i^{(2)},\ldots,\zeta_i^{(m)}].
	\end{equation*}
	This completes the proof.
\end{proof}
Now we are in a position to present  proofs of Theorem  \ref{main} and Theorem \ref{main2}.

\noindent\textbf{Proof of Theorem \ref{main}}  Note that $e_{mn}^\T=e_m^\T\otimes e_n^\T$ and $\eta_i^{j}=\zeta_i^{(j)}\otimes \xi_i$. By the mixed-product property of Kronecker products, we find that
$$e_{mn}^\T \eta_i^{(j)} =(e_m^\T\otimes e_n^\T)(\zeta_i^{(j)}\otimes \xi_i)=(e_m^\T\zeta_i^{(j)})(e_n^\T\xi_i).$$
Let $\Delta_1=\Res(\phi(M(H);x),\phi(M^{(v)}(H);x))$.
It follows from Corollary \ref{dwt3}, Lemmas \ref{Vanmu}, \ref{ppmu} and \ref{deteta1} that
\begin{align}\label{wvl}
	& \det W_M(G\circ H^{(v)}) \nonumber \\
	=& \frac{\prod\limits_{(i_1,j_1)<(i_2,j_2)}(\mu_{i_2}^{(j_2)}-\mu_{i_1}^{(j_1)})\prod\limits_{(i,j)\in \Omega}(e_{mn}^\T \eta_i^{(j)})}{\det[\eta_1^{(1)},\ldots,\eta_n^{(1)};\ldots;\eta_1^{(m)},\ldots,\eta_n^{(m)}]} \nonumber \\
	=& \pm \frac{\left(\prod\limits_{i=1}^{n}\prod\limits_{1\le j_1<j_2\le m}\left(\mu_i^{(j_2)}-\mu_i^{(j_1)}\right)\right)\left(\prod\limits_{1\le i_1< i_2\le n}(\lambda_{i_2}-\lambda_{i_1})^m\Delta_1\right)\prod\limits_{(i,j)\in \Omega}((e_m^\T\zeta_i^{(j)})(e_n^\T\xi_i))}{\left(\det[\xi_1,\xi_2,\ldots,\xi_n]\right)^m \prod\limits_{i=1}^{n}\det [\zeta_i^{(1)},\zeta_i^{(2)},\ldots,\zeta_i^{(m)}]} \nonumber \\
	=& \pm \Delta_1^{\frac{n(n-1)}{2}} \left( \prod\limits_{i=1}^n \frac{\prod\limits_{1\le j_1<j_2\le m}\left(\mu_i^{(j_2)}-\mu_i^{(j_1)}\right)\prod\limits_{j=1}^{m}(e_m^\T \zeta_i^{(j)})}{\det [\zeta_i^{(1)},\zeta_i^{(2)},\ldots,\zeta_i^{(m)}]} \right) \nonumber \\
	& \quad \times \left( \frac{\prod\limits_{1\le i_1< i_2\le n}(\lambda_{i_2}-\lambda_{i_1})\prod\limits_{1\le i\le n}(e_n^\T \xi_i)}{\det[\xi_1,\xi_2,\ldots,\xi_n]} \right)^m.
\end{align}
Let $\Delta_2$ and $\Delta_3$ denote the last two factors of Eq. \eqref{wvl}, respectively. 
Using Lemma \ref{basicW} for both $W(B(\lambda_i))$ and $W(M(G))$, we find that
\begin{equation*}
	\Delta_2=\prod_{i=1}^n \det W(B(\lambda_i))\text{~and~}	\Delta_3=(\det W(M(G)))^m.
\end{equation*}
Let $h(\lambda)=\det W(B(\lambda))$. Noting that the eigenvalues of  $h(M(G))$ are exactly $h(\lambda_1),\ldots,h(\lambda_n)$, we obtain that 
$$\Delta_2=\prod_{i=1}^n h(\lambda_i)=\det h(M(G)).$$
This completes the proof of Theorem \ref{main}. \hfill\qed

\noindent\textbf{Proof of Theorem \ref{main2}} Let $n$ be even and $G$ be any graph in $\mathcal{F}_n$.  By the definition of $\mathcal{F}_n$, we know that $\det A(G)=\pm 1$ and $\det W_A(G)=\pm 2^{n/2}$. Suppose that $H^{(v)}$ satisfies all the three conditions of Theorem \ref{main2}.  Clearly, by  Corollary \ref{suf}, we have $\det A(G\circ H^{(v)})=\pm 1$. Moreover, using Theorem \ref{main}, we find that 
\begin{equation*}
	\det W_A(G\circ H^{(v)})=\pm \left(\pm 1\right)^{\frac{n(n-1)}{2}}\cdot\det (\pm (A(G))^k)(\pm 2^{n/2})^m=\pm 2^{mn/2}.
\end{equation*}
Thus, $G\circ H^{(v)}\in \mathcal{F}$, i.e.,  $H^{(v)}$ is an $\mathcal{F}$-preserver. This completes the proof of Theorem \ref{main2}.
\section{Some examples and a conjecture}
As an illustration of our main results (Theorems \ref{main} and \ref{main2}), we consider a specific 4-vertex graph $H$ as shown in Fig.~\ref{gv4} and take the vertex $1$ as the root. 
\begin{figure}
	\centering
	\includegraphics[height=3.6cm]{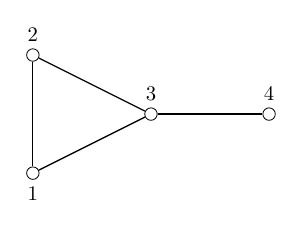}
	\caption{A 4-vertex graph $H$}
	\label{gv4}
\end{figure}
Direct calculations show that $\phi(A(H);x)=x^4-4 x^2-2 x+1$ and $\phi(A^{(1)}(H);x)=x^3-2 x$. The resultant of these two polynomials is equal to the determinant of their Sylvester matrix:

\begin{equation*}
\Res(\phi(A(H);x),\phi(A^{(1)}(H);x))=\det\begin{pmatrix}
		1 & 0 & -4 & -2 & 1 & 0 & 0 \\
		0 & 1 & 0 & -4 & -2 & 1 & 0 \\
		0 & 0 & 1 & 0 & -4 & -2 & 1 \\
		1 & 0 & -2 & 0 & 0 & 0 & 0 \\
		0 & 1 & 0 & -2 & 0 & 0 & 0 \\
		0 & 0 & 1 & 0 & -2 & 0 & 0 \\
		0 & 0 & 0 & 1 & 0 & -2 & 0 \\
\end{pmatrix}=1.
\end{equation*}
As \begin{equation*}
	B(\lambda)=A(H)+\lambda D_1=\left(
	\begin{array}{cccc}
		\lambda & 1 & 1 & 0 \\
		1 & 0 & 1 & 0 \\
		1 & 1 & 0 & 1 \\
		0 & 0 & 1 & 0 \\
	\end{array}
	\right),
\end{equation*}
we obtain 
\begin{equation*}
W(B(\lambda))=[e,B(\lambda)e,(B(\lambda))^2e,(B(\lambda))^3e]=\left(
\begin{array}{cccc}
	1 & \lambda+2 & \lambda^2+2 \lambda+5 & \lambda^3+2 \lambda^2+7 \lambda+10 \\
	1 & 2 & \lambda+5 & \lambda^2+3 \lambda+10 \\
	1 & 3 & \lambda+5 & \lambda^2+3 \lambda+13 \\
	1 & 1 & 3 & \lambda+5 \\
\end{array}
\right)
\end{equation*}
and hence $\det W(B(\lambda))=-\lambda^2$. It follows from  Theorem \ref{main} that, for this specific rooted graph $H^{(1)}$, we have
\begin{equation*}
	\det W_A(G\circ H^{(1)})=\pm (\det A(G))^2(\det W_A(G))^4.
\end{equation*}
Furthermore, as $\det A(H)=1$ and $\det A^{(1)}(H)=0$, we find that all conditions of Theorem \ref{main2} are satisfied. This implies that the graph $H^{(1)}$ is an $\mathcal{F}$-preserver. Thus, if $G$ is a graph in $\mathcal{F}$, then all graphs in the family
\begin{equation*}
	G, G\circ H^{(1)}, (G\circ H^{(1)})\circ H^{(1)},\ldots
\end{equation*} 
are in $\mathcal{F}$ and hence are DGS by the theorem of Wang \cite{wang2017JCTB}.

In \cite{yan2026}, Yan and Wang consider a family of graphs $H_{4t+1}$ of order $4t+1$, as illustrated in Fig. \ref{h4t}. 
\begin{figure}
	\centering
	\includegraphics[height=3cm]{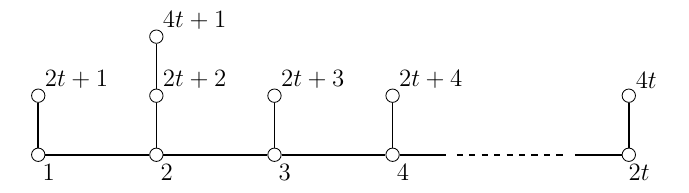}
	\caption{Graph $H_{4t+1}$}
	\label{h4t}
\end{figure}They conjectured that for $\ell =2t+1$ or  $4t+1$, it holds that 
\begin{equation*}
	\det W_A(G\circ H_{4t+1}^{(\ell)})=\pm (\det A(G))^{2t}(\det W_A(G))^{4t+1}.
\end{equation*}
(The two cases $\ell=2t+1$ and $\ell=4t+1$ are essentially the same as the corresponding two rooted graphs are isomorphic.)
Although a complete proof of this conjecture is not available, it can be verified computationally for any fixed integer $t$
based on Theorem 3, through a procedure similar to the one described above. 

In the Appendix, we list all $\mathcal{F}$-preservers with order at most 10 that are found by an exhaustive search based on Theorem \ref{main2}. The SageMath code for this search is available at https://github.com/wangweiAHPU/walk\_det. Remarkably, for all the listed $\mathcal{F}$-preservers $H^{(v)}$, the determinant of $W_A(G\circ H^{(v)})$  
exhibits a unified form dependent solely on the order of 
$H$. We propose the following conjecture for future investigation. 
\begin{conjecture}
	If $H^{(v)}$ is an $\mathcal{F}$-preserver of order $m$, then for any graph $G$, 
	\begin{equation*}
	\det W_A(G\circ H^{(v)})=\pm (\det A(G))^{\lfloor\frac{m}{2}\rfloor}(\det W_A(G))^m.
	\end{equation*}
\end{conjecture}

\section*{Declaration of competing interest}
There is no conflict of interest.
\section*{Acknowledgments}
This work is partially supported by the National Natural Science Foundation of China (Grant Nos. 12001006 and 12101379) and  Wuhu Science and Technology Project, China (Grant No.~2024kj015).

\clearpage
\section*{Appendix}

\setlength{\tabcolsep}{0pt} 

\renewcommand{\arraystretch}{0.1} 

\def\TableContent{}
\foreach \n in {1,...,97}{
	\ifnum\n<10
	\xdef\CurrentFile{graph_pdfs/F0\n.pdf}
	\else
	\xdef\CurrentFile{graph_pdfs/F\n.pdf}
	\fi
	\xappto\TableContent{\noexpand\includegraphics[width=\linewidth]{\CurrentFile}}
	\ifnum\n=97
	\else
	\ifnum\numexpr\n-5*(\n/5)\relax=0
	\gappto\TableContent{\\ \noalign{\vspace{1pt}}} 
	\else
	\gappto\TableContent{&} 
	\fi
	\fi
}

\begin{longtable}{*{5}{>{\centering\arraybackslash}p{0.19\textwidth}}}
	
	\caption{A list of $\mathcal{F}$-preservers of order $\le 10$} \label{tab:f-preservers}\\ 
	\multicolumn{5}{c}{\small (Solid circles denote root  vertices)} \\[2ex] 
	\endfirsthead
	
	
	\caption[]{A list of $\mathcal{F}$-preservers of order $\le 10$ (Continued)} \\
	\endhead
	
	\multicolumn{5}{r}{\scriptsize\textit{(To be continued on next page)}} \\
	\endfoot
	
	\endlastfoot
	
	\TableContent
	
\end{longtable}
\end{document}